\documentclass[11pt,a4paper]{article}
\usepackage[T1]{fontenc}
\usepackage[utf8]{inputenc}
\usepackage{lmodern}
\usepackage{microtype}
\usepackage{amsmath,amssymb,amsthm,mathtools,bm}
\usepackage{geometry}
\usepackage{enumitem}
\usepackage{xcolor}
\usepackage{cite}
\usepackage[hidelinks]{hyperref}
\usepackage[nameinlink,capitalise,noabbrev]{cleveref}
\allowdisplaybreaks

\newtheorem{theorem}{Theorem}[section]
\newtheorem{proposition}[theorem]{Proposition}
\newtheorem{lemma}[theorem]{Lemma}
\newtheorem{corollary}[theorem]{Corollary}
\theoremstyle{definition}
\newtheorem{definition}[theorem]{Definition}

\theoremstyle{remark}
\newtheorem{remark}[theorem]{Remark}

\newcommand{\A}{\mathbb A}
\newcommand{\R}{\mathbb R}
\newcommand{\m}{\mathfrak m}
\newcommand{\one}{1_{\A}}
\newcommand{\pr}{\pi}
\newcommand{\JA}{J_{\A}}
\newcommand{\so}{\mathfrak{so}}
\newcommand{\SO}{\mathrm{SO}}
\newcommand{\DR}{\mathrm D_{\!R}}
\newcommand{\Ann}{\operatorname{Ann}}
\newcommand{\Exp}{\operatorname{Exp}_{\A}}
\newcommand{\Log}{\operatorname{Log}_{\A}}
\newcommand{\wide}[1]{[\,#1\,]_{\times}}
\hypersetup{
  pdftitle={Finite Weil Jets, Identity Transfer, and the Exponential Geometry of Proper Orthogonal Tensors},
  pdfauthor={Daniel Condurache},
  pdfsubject={Finite Weil jets and exponential geometry over Weil algebras},
  pdfkeywords={Weil algebra, finite jet, Weil functor, proper orthogonal group, Rodrigues formula, Lie group exponential}
}

\title{Finite Weil Jets, Identity Transfer, and the Exponential Geometry of Proper Orthogonal Tensors}
\author{Daniel Condurache\\
\small Department of Theoretical Mechanics, Gheorghe Asachi Technical University of Ia\c si,\\
\small D. Mangeron Street 59, 700050 Ia\c si, Romania\\
\small Corresponding author: \href{mailto:daniel.condurache@tuiasi.ro}%
{\texttt{daniel.condurache@tuiasi.ro}}}
\date{}

\begin{document}
\maketitle

\begin{abstract}
Let \(\A=\R\one\oplus\m\) be a Weil algebra with \(\m^{N+1}=0\). We study the
exponential geometry of the proper orthogonal group \(\SO(3,\A)\) through finite Weil
prolongation. An exact Rodrigues formula is obtained for every element of
\(\so(3,\A)\), and the exponential map \(\so(3,\A)\to\SO(3,\A)\) is proved to be
surjective. Although every proper orthogonal Weil tensor is therefore exponential and has an
unnormalised Rodrigues representation, a factorisation of a generator into one scalar angle
and one unit axis may fail above the identity; the obstruction is established for every
logarithmic branch, including scalar angles \(2k\pi\). The analytic mechanism is an explicit
finite-regularity rigidity theorem: an \(\A\)-holomorphic map on \(U+\m\) with scalar-valued
real trace is uniquely its finite Weil jet. Together with functoriality, this yields exact
transfer of functional, matrix, orthogonality, and composition identities. The results provide
a rigorous basis for dual, multidual, hyper-dual, and hyper-multidual calculations in
theoretical and computational kinematics.
\end{abstract}

\noindent\textbf{Keywords:} Weil algebra; finite jet; Weil functor; proper orthogonal group;
Rodrigues formula; Lie group exponential; multidual kinematics.

\noindent\textbf{Mathematics Subject Classification (2020):} 13N99; 30G35; 58A20; 53A17.

\section{Introduction}

Nilpotent extensions of the real numbers occur in several mathematical and computational
languages that developed partly independently. Dual numbers encode first-order differential
information and have a classical role in screw theory and spatial kinematics. Higher-order
chain algebras encode successive derivatives; hyper-dual algebras isolate mixed second
derivatives; tensor products of such algebras support multivariate automatic differentiation;
and hyper-multidual algebras combine higher-order differentiation with the dual structures
used for rigid displacements. These constructions are special cases of finite-dimensional
local algebras introduced geometrically by Weil through his theory of near points. Their
prolongation to manifolds, functoriality, and interpretation as scalar extension belong to the
classical theory of Weil functors and manifolds over local algebras
\cite{Weil1953,Kolar1993,BertramSouvay2014,Bertram2014,Shurygin1993,Shurygin2002}.

In computational kinematics, dual algebra replaces pairs of rotational and translational
relations by a single algebraic relation and supports invariant descriptions of lines, screws,
and rigid displacements \cite{Angeles1998,BottemaRoth1990}. It also leads to effective linear
algebra and numerical algorithms \cite{PennestriStefanelli2007}. Hyper-dual arithmetic was
introduced for exact second-derivative calculations \cite{FikeAlonso2011} and developed for
numerical optimisation \cite{FikeEtAl2011}. Multidual and hyper-multidual structures
subsequently provide compact
representations of velocity, acceleration, jerk, snap, and arbitrary finite orders of rigid-body
and robot kinematics \cite{Condurache2020,Condurache2022,ConduracheOverview2025}. Their computational use was
examined in \emph{Mechanism and Machine Theory} for a surgical parallel robot by comparing
classical differentiation and multidual implementations in vector, homogeneous-matrix, and
dual-quaternion formalisms \cite{ConduracheMMT2025}. The common mechanism is not an infinite
series in a nilpotent variable: it is a finite jet whose order is fixed by the nilpotency index
of the coefficient algebra.

This distinction matters. Expressions such as \(\sin Z\), \(\exp Z\), or \(f(x+H)\) are often
written by formally substituting a dual or multidual argument into a real Taylor series. Such
notation may return the correct finite formula, but it is not itself a definition of
differentiability, holomorphy, or Taylor expansion on a Weil algebra. Here all prolongations
are defined by finite real jets. Taylor's theorem enters only once, for a function of a real
parameter with values in a finite-dimensional real vector space.

The abstract existence and functoriality of Weil prolongations are classical. The present
article develops an explicit finite-regularity calculus on open subsets of real vector spaces
and uses it to analyse a concrete Lie-theoretic question. The scalar-valued real trace rigidly
determines an \(\A\)-holomorphic function and forces it to be the finite jet of that trace;
finite prolongation then preserves sums, products, inverses, admissible compositions, and
identities. Applied to rotations, this mechanism yields an exact Rodrigues calculus and the
principal geometric results of the paper: the exponential map onto \(\SO(3,\A)\) is
surjective, whereas a normalised scalar-angle/unit-axis factorisation may fail above the
identity. The branch-independent obstruction in \cref{thm:noaxis} makes precise the
difference between exponentiality and unit-axis factorisation.

\section{Weil algebras and finite prolongation}

\begin{definition}
A \emph{Weil algebra} is a finite-dimensional, commutative, associative real algebra with unit
of the form
\[
  \A=\R\one\oplus\m,
  \qquad \m^{N+1}=0,\quad \m^N\ne0,
\]
where \(\m\) is the unique maximal ideal. The integer \(N\) is its nilpotency index. The
canonical algebra morphism \(\pr:\A\to\R\) is the scalar-part projection and
\(\ker\pr=\m\).
\end{definition}

Every \(Z\in\A\) is uniquely \(Z=x+H\), with \(x=\pr(Z)\) and \(H\in\m\). Relevant examples
include
\begin{align*}
 \R[\varepsilon]/(\varepsilon^2)& &&(N=1),\\
 \R[\varepsilon]/(\varepsilon^{n+1})& &&(N=n),\\
 \R[\varepsilon_1,\varepsilon_2]/(\varepsilon_1^2,\varepsilon_2^2)& &&(N=2),\\
 \R[\varepsilon_1,\ldots,\varepsilon_n]/(\varepsilon_1^2,\ldots,\varepsilon_n^2)& &&(N=n),\\
 \R[\varepsilon]/(\varepsilon^{n+1})\otimes_{\R}
 \R[\varepsilon_0]/(\varepsilon_0^2)& &&(N=n+1).
\end{align*}
They are, respectively, the dual, chain multidual, two-direction hyper-dual,
\(2^n\)-dimensional square-zero-generator, and hyper-multidual algebras.

Let \(U\subset\R^p\) be open and put
\(U^{\A}=\{X\in\A^p:\pr(X)\in U\}\), where \(\pr\) acts componentwise.

\begin{definition}[Finite Weil prolongation]
For \(f\in C^N(U,\R^q)\), \(x\in U\), and
\(H=(H_1,\ldots,H_p)\in\m^p\), define
\begin{equation}\label{eq:prolongation}
 (\JA f)(x+H)=
 \sum_{|\alpha|\le N}\frac{1}{\alpha!}
 (\partial^\alpha f)(x)H^\alpha\in\A^q.
\end{equation}
The sum is finite. Terms of degree greater than \(N\) vanish because
\(H^\alpha\in\m^{|\alpha|}=0\).
\end{definition}

Thus \eqref{eq:prolongation} is an exact algebraic evaluation of a finite jet, not a convergent
or formally truncated infinite series. It is independent of coordinates in the standard
functorial sense of Weil prolongation \cite[Chapter VIII]{Kolar1993}.

\section{Holomorphy and uniqueness in one variable}

For the rigidity theorem we first take \(p=q=1\), let \(U\subset\R\) be an open interval,
and identify \(\A\) with its underlying finite-dimensional real vector space.

\begin{definition}[\(\A\)-holomorphy]
A map \(F\in C^1(U^{\A},\A)\) is \emph{\(\A\)-holomorphic} if its real Fr\'echet differential
is \(\A\)-linear at every point. Equivalently, there is a uniquely determined map
\(F':U^{\A}\to\A\) such that
\begin{equation}\label{eq:hol}
 \DR F(Z)[K]=F'(Z)K,\qquad K\in\A.
\end{equation}
Indeed, \(F'(Z)=\DR F(Z)[\one]\).
\end{definition}

This definition involves only the first real differential and is neither circular nor dependent
on an undeclared multidual Taylor expansion.

\begin{lemma}[Propagation of holomorphy]\label{lem:propagation}
If \(r\ge2\) and \(F\in C^r(U^{\A},\A)\) is \(\A\)-holomorphic, then
\(F'=\DR F(\cdot)[\one]\in C^{r-1}\) is \(\A\)-holomorphic. More precisely,
\[
 \DR F'(Z)[K]=F''(Z)K,
 \qquad F''(Z):=\DR F'(Z)[\one].
\]
\end{lemma}

\begin{proof}
For \(K,L\in\A\), differentiate \eqref{eq:hol} and use the symmetry of the second real
differential:
\[
 \DR^2F(Z)[K,L]=\DR F'(Z)[L]K
                 =\DR F'(Z)[K]L.
\]
Set \(K=\one\). Then
\(\DR F'(Z)[L]=\DR F'(Z)[\one]L\), as required.
\end{proof}

Iterating \cref{lem:propagation}, if \(F\in C^{N+1}\), the functions
\(F^{(0)},\ldots,F^{(N)}\) are holomorphic, and
\(F^{(N+1)}:=\DR F^{(N)}(\cdot)[\one]\) is continuous.

\begin{theorem}[Uniqueness of the finite-jet extension]\label{thm:uniqueness}
Let \(\A\) have nilpotency index \(N\). Suppose that
\(F\in C^{N+1}(U^{\A},\A)\) is \(\A\)-holomorphic and that its restriction to the real axis
is scalar-valued:
\begin{equation}\label{eq:trace}
 F(x)=f(x)\one,\qquad x\in U.
\end{equation}
Then \(f\in C^{N+1}(U,\R)\), and
\begin{equation}\label{eq:uniquejet}
 F(x+H)=\sum_{k=0}^{N}\frac{f^{(k)}(x)}{k!}H^k
       =(\JA f)(x+H),\qquad H\in\m.
\end{equation}
Consequently, two \(C^{N+1}\), \(\A\)-holomorphic maps with the same scalar-valued trace
coincide.
\end{theorem}

\begin{proof}
First, induction along the real axis gives
\begin{equation}\label{eq:axisderivatives}
 F^{(k)}(x)=f^{(k)}(x)\one,\qquad 0\le k\le N+1.
\end{equation}
For \(k=0\) this is \eqref{eq:trace}. If it holds at rank \(k\le N\), then
\[
 F^{(k+1)}(x)=\DR F^{(k)}(x)[\one]
 =\lim_{t\to0}\frac{f^{(k)}(x+t)-f^{(k)}(x)}{t}\,\one.
\]
The left-hand side exists, so the real derivative exists and the equality follows. Continuity
at the last rank follows by restricting the continuous map \(F^{(N+1)}\), proving
\(f\in C^{N+1}\).

Fix \(x\in U\), \(H\in\m\), and consider only the real-parameter curve
\(\Phi(t)=F(x+tH)\), \(0\le t\le1\). The ordinary real chain rule and
\cref{lem:propagation} give
\[
 \Phi^{(k)}(t)=F^{(k)}(x+tH)H^k,\qquad 0\le k\le N+1.
\]
Since \(H^{N+1}=0\), one has \(\Phi^{(N+1)}\equiv0\). Taylor's theorem with integral
remainder, applied componentwise to the finite-dimensional real-vector-valued function
\(\Phi\), therefore gives
\[
 \Phi(1)=\sum_{k=0}^{N}\frac{\Phi^{(k)}(0)}{k!};
\]
the integral remainder is identically zero. Substitution of
\eqref{eq:axisderivatives} proves \eqref{eq:uniquejet}. The last assertion follows pointwise.
\end{proof}

\begin{remark}[Exact location of Taylor's theorem]\label{rem:noTaylor}
Taylor's theorem in the proof is applied to \(t\mapsto F(x+tH)\), whose variable \(t\) is real.
No Taylor theorem for an \(\A\)-valued variable is assumed. Nilpotency does not establish
convergence of an infinite expansion; it annihilates the last derivative of this real curve and
makes the real Taylor formula exact.
\end{remark}

\begin{proposition}[Existence]\label{prop:existence}
If \(f\in C^{N+1}(U,\R)\), then \(\JA f\in C^1(U^{\A},\A)\) is
\(\A\)-holomorphic and
\begin{equation}\label{eq:derivativejet}
 \DR(\JA f)(Z)[K]=(\JA f')(Z)K.
\end{equation}
More generally, \(\JA f\in C^r\) whenever \(f\in C^{N+r}\).
\end{proposition}

\begin{proof}
Differentiate the finite expression \eqref{eq:prolongation} in the underlying real vector
space. Write \(K=\pr(K)\one+K_0\) with \(K_0\in\m\). After the scalar and nilpotent
contributions are combined, the only unmatched terms have total nilpotent degree \(N+1\),
and hence vanish. The remaining expression is exactly the right-hand side of
\eqref{eq:derivativejet}. Repetition proves the regularity statement.
\end{proof}

In particular, if existence and the \(C^{N+1}\) uniqueness class are required simultaneously,
the sufficient scalar regularity is \(f\in C^{2N+1}\). On \(C^\infty\), restriction and finite
prolongation are inverse bijections between smooth scalar-trace \(\A\)-holomorphic functions
and smooth real functions. The finite regularity assumed in \cref{thm:uniqueness} is sufficient
for the stated proof; no optimality claim is made here.

\begin{remark}[Why scalar-valued trace is essential]
Holomorphy alone is insufficient. Choose \(0\ne C\in\Ann(\m)\), which is possible because
\(\m^N\subseteq\Ann(\m)\). For any \(g\in C^1(U,\R)\),
\[
 G(Z)=g(\pr Z)C
\]
is \(\A\)-holomorphic: \(\DR G(Z)[K]=g'(\pr Z)\pr(K)C=g'(\pr Z)CK\).
Its trace \(g(x)C\) is generally not scalar-valued. This is the nilpotent freedom already
visible in the classical Scheffers representation \cite{Scheffers1893}.
\end{remark}

\section{Products, compositions, and transfer of identities}

The finite prolongation is functorial. This supplies the precise conditions under which a
composition is preserved.

\begin{theorem}[Finite-jet calculus]\label{thm:functoriality}
Let \(U\subset\R^p\) and \(V\subset\R^q\) be open. Then:
\begin{enumerate}[label=\textup{(\roman*)}]
\item for \(f,h\in C^N(U,\R)\),
\(\JA(f+h)=\JA f+\JA h\) and \(\JA(fh)=(\JA f)(\JA h)\);
\item if, in addition, \(f(x)\ne0\) on \(U\), then
\(\JA(f^{-1})=(\JA f)^{-1}\);
\item if \(f\in C^N(U,\R^q)\), \(f(U)\subset V\), and
      \(g\in C^N(V,\R^r)\), then, for every \(X\in U^{\A}\),
\begin{equation}\label{eq:composition}
 \JA(g\circ f)(X)=(\JA g)((\JA f)(X)).
\end{equation}
\end{enumerate}
For vector- and matrix-valued maps, these statements hold componentwise, with every product
formed in the indicated finite-dimensional associative algebra.
\end{theorem}

\begin{proof}
Addition is immediate. The multivariate Leibniz rule makes the coefficients of
\(\JA(fh)\) equal to those in the finite Cauchy product of \(\JA f\) and \(\JA h\); all terms
of nilpotent degree exceeding \(N\) vanish. Applying multiplicativity to \(ff^{-1}=1\) proves
(ii); invertibility follows because an element of \(\A\) is invertible precisely when its scalar
part is nonzero.

For (iii), at each real base point take the ordinary degree-\(N\) real Taylor polynomial of
\(g\) and substitute the finite jet of \(f\). The multivariate chain rule (equivalently, the
Fa\`a di Bruno formula) identifies every coefficient of total degree at most \(N\) with the
corresponding derivative of \(g\circ f\); higher-degree terms lie in \(\m^{N+1}\) and vanish.
Only finite polynomials and derivatives of real-variable maps occur in this argument.
\end{proof}

Notice that \((\JA f)(X)\in V^{\A}\) is automatic, because
\(\pr((\JA f)(X))=f(\pr X)\in V\). For a locally defined or multivalued outer function, one
single-valued branch must be fixed on \(V\); with this understood, the composition law applies
without an additional restriction on the nilpotent part.

\begin{corollary}[Exact transfer of identities]\label{cor:identities}
Let \(P:\R^s\to\R^r\) be a polynomial map and let
\(f_1,\ldots,f_s\in C^N(U)\). If
\[
 P(f_1(x),\ldots,f_s(x))=0\qquad(x\in U),
\]
then
\[
 P(\JA f_1(X),\ldots,\JA f_s(X))=0\qquad(X\in U^{\A}).
\]
More generally, let \(f=(f_1,\ldots,f_s)\), let \(V\subset\R^s\) be open with
\(f(U)\subset V\), and let \(h\in C^N(V,\R^r)\) be a fixed single-valued map. Then
\[
 h\circ f=0\quad\Longrightarrow\quad
 (\JA h)\circ(\JA f)=\JA(h\circ f)=0.
\]
Thus a nonpolynomial outer map is evaluated through its finite prolongation; any required
branch is fixed on \(V\) before prolongation.
\end{corollary}

This is the precise conservation principle used in kinematics. Orthogonality, unit-quaternion
constraints, determinant identities, matrix inverse identities, and the composition laws of
motions are preserved because they are identities built from finite algebraic operations and
admissible compositions. For a time-dependent family, the additional requirement that its
coefficients are the successive derivatives of one real motion is expressed as follows.

\begin{definition}[Temporal holonomicity of a family]\label{def:holonomicity}
Let \(G\subset\R^r\) be a smooth matrix group, let \(I\subset\R\) be an open interval,
and put \(\A_N=\R[\varepsilon]/(\varepsilon^{N+1})\). Consider a family
\[
 X(t)=\sum_{k=0}^{N}\varepsilon^kX_k(t),\qquad
 X_k\in C^{N-k}(I,\R^r),\quad X_0(t)\in G.
\]
The family is \emph{temporally holonomic} if
\[
 X_k(t)=\frac{1}{k!}X_0^{(k)}(t)\quad(0\le k\le N),
 \qquad\text{equivalently}\qquad
 X(t)=(J_{\A_N}X_0)(t+\varepsilon).
\]
Equivalently, its coefficients satisfy
\(\dot X_k=(k+1)X_{k+1}\) for \(0\le k<N\).
\end{definition}

For example, if \(0\ne B\in\so(3,\R)\) and \(\varepsilon^2=0\), then
\(Q(t)=I+\varepsilon B\) belongs to \(\SO(3,\A_1)\) for every \(t\), but this family
is not temporally holonomic: its real part is the constant curve \(I\), whose derivative
is zero. This is a compatibility condition on a family, not an obstruction to realising
an individual point as a curve jet; the latter is made explicit for rotations in
\cref{sec:kinematics}.

\section{Finite functional calculus over a Weil algebra}

For a real \(C^N\) function \(f:I\to\R\), its value at \(Z=x+H\in I^{\A}\) is, by definition,
\((\JA f)(Z)\). Thus the usual elementary functions have the following finite expressions:
\begin{align}
 \exp_{\A}(x+H)&=e^x\sum_{k=0}^{N}\frac{H^k}{k!},\label{eq:expfinite}\\
 \sin_{\A}(x+H)&=\sum_{k=0}^{N}\frac{\sin(x+k\pi/2)}{k!}H^k,\\
 \cos_{\A}(x+H)&=\sum_{k=0}^{N}\frac{\cos(x+k\pi/2)}{k!}H^k,\\
 (x+H)^{-1}&=x^{-1}\sum_{k=0}^{N}(-H/x)^k,\qquad x\ne0,\\
 \log_{\A}(x+H)&=\log x+
 \sum_{k=1}^{N}\frac{(-1)^{k-1}}{k\,x^k}H^k,\qquad x>0,\\
 (x+H)^\alpha&=\sum_{k=0}^{N}
 \frac{\alpha(\alpha-1)\cdots(\alpha-k+1)}{k!}
 x^{\alpha-k}H^k,\quad x>0.\label{eq:powerfinite}
\end{align}
The empty product in \eqref{eq:powerfinite} is one. Every displayed sum is finite. The
logarithm and real power require the stated domain; a single branch must be selected before
prolongation whenever one is required.

By \cref{thm:functoriality}, familiar identities are exact whenever their real domains are
respected, for example
\[
 \sin_{\A}^2 Z+\cos_{\A}^2Z=1,
 \qquad \exp_{\A}(Z+W)=\exp_{\A}(Z)\exp_{\A}(W),
\]
the second identity using commutativity of \(\A\). These are consequences of finite-jet
functoriality.

\section{Rodrigues formula over a Weil algebra}

Let \(M_3(\A)\) denote the \(3\times3\) matrices over \(\A\), and
\[
 \so(3,\A)=\{B\in M_3(\A):B^T=-B\},\qquad
 \SO(3,\A)=\{Q:Q^TQ=I,\ \det Q=1\}.
\]
Throughout this section, \(\Exp\) denotes the finite Weil prolongation of the real matrix
exponential. Whenever a single real branch of the matrix logarithm is fixed on an open
neighbourhood, \(\Log\) denotes its finite Weil prolongation.

For \(\phi=(\phi_1,\phi_2,\phi_3)^T\in\A^3\), let
\[
 \wide{\phi}=\begin{pmatrix}
 0&-\phi_3&\phi_2\\ \phi_3&0&-\phi_1\\-\phi_2&\phi_1&0
 \end{pmatrix}.
\]
Because \(\A\) is commutative, the usual cross-product calculation remains valid:
\begin{equation}\label{eq:cubic}
 \wide{\phi}^{2}=\phi\phi^T-(\phi^T\phi)I,
 \qquad \wide{\phi}^{3}=-(\phi^T\phi)\wide{\phi}.
\end{equation}

For real \(\phi=\theta u\), where \(\|u\|=1\), define the left Jacobian \(J(\phi)\) by
right-trivialising the differential of the exponential:
\begin{equation}\label{eq:leftJacobianDefinition}
 \begin{aligned}
 (\mathrm d\exp)_{\wide\phi}(\wide c)\exp(-\wide\phi)
   &=\wide{J(\phi)c},\\
 J(\theta u)
   &=I+\frac{1-\cos\theta}{\theta^2}\,\wide{\theta u}
     +\frac{\theta-\sin\theta}{\theta^3}\,\wide{\theta u}^{2},
 \end{aligned}
\end{equation}
with the continuous value \(J(0)=I\). In particular,
\begin{equation}\label{eq:JacobianConsequences}
 \det J(\theta u)=\left(\frac{2\sin(\theta/2)}{\theta}\right)^2,
 \qquad J(2k\pi u)=uu^T\quad(k\in\mathbb Z\setminus\{0\}).
\end{equation}

Define real functions \(a,b:\R\to\R\) by
\begin{align*}
 a(s)&=\begin{cases}
 \sin\sqrt{s}/\sqrt{s},&s>0,\\1,&s=0,\\
 \sinh\sqrt{-s}/\sqrt{-s},&s<0,
 \end{cases}\\
 b(s)&=\begin{cases}
 (1-\cos\sqrt{s})/s,&s>0,\\1/2,&s=0,\\
 (\cosh\sqrt{-s}-1)/(-s),&s<0.
 \end{cases}
\end{align*}
These functions are smooth on \(\R\). One proof involving no power series is to let
\(c_s(t)\) be the solution of
\[
 \partial_t^2c_s+s c_s=0,\qquad c_s(0)=1,\quad \partial_t c_s(0)=0.
\]
Smooth dependence of this real initial-value problem on the parameter \(s\), together with
\[
 a(s)=\int_0^1c_s(t)\,dt,
 \qquad b(s)=\int_0^1(1-t)c_s(t)\,dt,
\]
proves the assertion and gives the displayed trigonometric and hyperbolic forms, including
their values at zero. For \(S=s_0+K\in\A\), define
\(a_{\A}(S)=\JA a(S)\) and
\(b_{\A}(S)=\JA b(S)\) by \eqref{eq:prolongation}. This definition again uses only finite jets.

\begin{theorem}[Exact Weil--Rodrigues formula]\label{thm:Rodrigues}
For every \(B=\wide{\phi}\in\so(3,\A)\), with \(S=\phi^T\phi\), the matrix exponential
obtained by finite Weil prolongation of the real matrix exponential satisfies
\begin{equation}\label{eq:Rodrigues}
 \Exp(B)=I+a_{\A}(S)B+b_{\A}(S)B^2\in\SO(3,\A).
\end{equation}
This is an exact algebraic identity in \(M_3(\A)\).
\end{theorem}

\begin{proof}
For a real vector \(v\), the ordinary Rodrigues identity is
\[
 \exp(\wide v)=I+a(v^Tv)\wide v+b(v^Tv)\wide v^{2}.
\]
Both sides are smooth real matrix-valued functions of the three coordinates of \(v\).
Apply the finite prolongation and use \cref{thm:functoriality,cor:identities}. This gives
\eqref{eq:Rodrigues} without introducing a matrix series over \(\A\). Prolonging the real
identities \(R^TR=I\) and \(\det R=1\) gives \(\Exp(B)\in\SO(3,\A)\).
\end{proof}

Equation \eqref{eq:cubic} explains algebraically why only \(I,B,B^2\) can occur, but it does
not by itself identify the coefficients. Their identification in \eqref{eq:Rodrigues} comes from
the finite prolongation of the ordinary real Rodrigues identity. Thus the coefficients are
determined directly by identity transfer, without assigning an infinite series to a multidual
matrix variable.

\begin{corollary}[Unit-axis form for a nonzero real generator]\label{cor:axis}
Suppose that \(B=\wide{\phi}\), let \(S=\phi^T\phi\), and assume that
\(s_0=\pr(S)=\pr(\phi)^T\pr(\phi)>0\). Put
\[
 \theta=\JA(\sqrt{\cdot})(S),\qquad u=\theta^{-1}\phi.
\]
Then \(\theta\) is invertible, \(u^Tu=1\), and
\begin{equation}\label{eq:unitRodrigues}
 \Exp(B)=I+\sin_{\A}(\theta)\wide u+
 (1-\cos_{\A}(\theta))\wide u^{2}.
\end{equation}
In particular, for a chosen real generator with
\(0<\theta_0=\pr(\theta)=\sqrt{s_0}\le\pi\), this normalisation is regular,
including at \(\theta_0=\pi\).
\end{corollary}

\begin{proof}
Since \(\pr(\theta)=\sqrt{s_0}>0\), \(\theta\) is invertible in \(\A\). Functoriality gives
\(\theta^2=S\), whence \(u^Tu=S/\theta^2=1\). Prolonging the real identities
\[
 a(t^2)t=\sin t,
 \qquad b(t^2)t^2=1-\cos t \qquad (t>0)
\]
gives their \(\A\)-valued counterparts. Substitution into \eqref{eq:Rodrigues} proves
\eqref{eq:unitRodrigues}.
\end{proof}

\section{Every proper orthogonal Weil tensor is exponential}

Entrywise prolongation identifies \(\SO(3,\A)\) with the \(\A\)-points of the real
algebraic Lie group \(\SO(3)\), and its Lie algebra is
\(\so(3,\A)=\so(3,\R)\otimes_{\R}\A\). Reduction modulo \(\m\), together with the
constant inclusion, gives the split exact sequence
\begin{equation}\label{eq:reductionSequence}
 1\longrightarrow K_{\A}\longrightarrow\SO(3,\A)
 \xrightarrow{\,\pr_{\A}\,}\SO(3,\R)\longrightarrow1,
 \qquad K_{\A}=\{Q:\pr_{\A}(Q)=I\}.
\end{equation}
This is the standard Weil--Lie group structure associated with a Weil algebra
\cite{Bertram2014}. The following theorem determines its exponential geometry.

\begin{theorem}[Surjectivity of the exponential]\label{thm:surjectivity}
For every \(Q\in\SO(3,\A)\), there exists \(B\in\so(3,\A)\) such that
\(Q=\Exp(B)\). Hence every proper orthogonal \(3\times3\) tensor over a Weil algebra has an exact
unnormalised Rodrigues representation \eqref{eq:Rodrigues}.
\end{theorem}

\begin{proof}
Put \(Q_0=\pr(Q)\) and choose a real generator
\(B_0=\wide{\phi_0}\in\so(3,\R)\) with
\(\exp B_0=Q_0\) and \(\theta_0=\|\phi_0\|\in[0,\pi]\). In axial coordinates, the
right-trivialised differential of the real exponential is the Jacobian
\(J(\phi_0)\) defined in \eqref{eq:leftJacobianDefinition}. By
\eqref{eq:JacobianConsequences}, its determinant is nonzero for
\(0\le\theta_0\le\pi\), with value one at \(\theta_0=0\), including at
\(\theta_0=\pi\).

Let \(\operatorname{Sym}(3,\R)\) be the space of real symmetric matrices and consider the
smooth map between real vector spaces of equal dimension
\[
 \Psi:\so(3,\R)\times\operatorname{Sym}(3,\R)\longrightarrow M_3(\R),
 \qquad \Psi(B,S)=\exp(B)\exp(S).
\]
After right translation by \(Q_0^{-1}\), its differential at \((B_0,0)\) sends
\((\wide c,T)\) to
\[
 \wide{J(\phi_0)c}+Q_0TQ_0^T.
\]
The first term is skew-symmetric and the second symmetric; moreover,
\(T\mapsto Q_0TQ_0^T\) is an automorphism of \(\operatorname{Sym}(3,\R)\). Since
\(J(\phi_0)\) is invertible
by \eqref{eq:JacobianConsequences} and
\(M_3(\R)=\so(3,\R)\oplus\operatorname{Sym}(3,\R)\), this differential is an
isomorphism. The inverse function theorem therefore gives open neighbourhoods of
\((B_0,0)\) and \(Q_0\) on which \(\Psi\) is a diffeomorphism.

Because \(\pr(Q)=Q_0\), the Weil point \(Q\) belongs to the prolongation of the latter open
neighbourhood. Define
\[
 (B,S)=(\JA\Psi^{-1})(Q)
 \in\so(3,\A)\times\operatorname{Sym}(3,\A).
\]
For real \(M=\Psi(C,T)\) in this neighbourhood,
\[
 M^TM=\exp(2T),
 \qquad T=\frac12\log(M^TM),
\]
where the matrix logarithm is the smooth real logarithm on the positive-definite matrices near
\(I\). Prolongation of this identity and \(Q^TQ=I\) give
\[
 S=\frac12\Log(Q^TQ)=0.
\]
Finally, prolonging \(\Psi\circ\Psi^{-1}=\mathrm{id}\) yields
\[
 Q=\Exp(B)\Exp(S)=\Exp(B),
\]
which proves the assertion using only maps defined on open subsets of real vector spaces.
\end{proof}

When \(Q_0\ne I\), the lift \(B\) associated with the chosen real generator \(B_0\)
in \cref{thm:surjectivity} has
\(s_0\in(0,\pi^2]\), and \cref{cor:axis} constructs explicitly
\(\theta=\JA(\sqrt{\cdot})(\phi^T\phi)\) and \(u=\phi\theta^{-1}\), with
\(u^Tu=1\). Thus the normalised axis-angle formula is available on every nonidentity real
fibre and is smooth on the chosen local branch. At a half-turn, put \(u_0=\phi_0/\pi\).
Either of the two real generators \(\pm\pi\wide{u_0}\) may be chosen, and each gives
a regular local branch.

If \(Q_0=I\), a particularly direct principal logarithm is available. Since
\(X=Q-I\in\m\otimes M_3(\R)\) is nilpotent,
\begin{equation}\label{eq:finiteLog}
 B=\Log Q:=\sum_{j=1}^{N}\frac{(-1)^{j+1}}{j}X^j
\end{equation}
is a finite polynomial. The finite polynomial identities for logarithm and exponential give
\(\Exp(B)=Q\). Indeed, at scalar base point zero the finite prolongation of the real matrix
exponential is exactly \(\sum_{j=0}^{N}B^j/j!\); hence this polynomial calculation and the
Weil-prolonged matrix exponential are the same map on the identity fibre. Moreover,
\[
 B^T=\Log(Q^T)=\Log(Q^{-1})=-\Log Q,
\]
so \(B\in\so(3,\A)\). Formula \eqref{eq:finiteLog} is a finite nilpotent identity, not a
convergence assertion.

\section{Failure of a unit axis above the identity}

The existence of an exponential generator does not imply that the generator factors as a
single scalar angle times a unit multidual axis.

\begin{theorem}[A branch-independent obstruction]\label{thm:noaxis}
Let
\[
 \A=\R[\varepsilon_1,\varepsilon_2]/
 (\varepsilon_1^2,\varepsilon_2^2)
\]
and choose linearly independent \(p,q\in\R^3\). Set
\[
 \phi=\varepsilon_1p+\varepsilon_2q,
 \qquad Q=\Exp(\wide\phi)\in\SO(3,\A).
\]
Then there are no \(\theta\in\A\) and \(u\in\A^3\) such that
\[
 u^Tu=1,\qquad Q=\Exp(\theta\wide u).
\]
The conclusion covers both the principal branch \(\pr(\theta)=0\) and every nonprincipal
branch \(\pr(\theta)=2k\pi\), \(k\ne0\).
\end{theorem}

\begin{proof}
Assume such \(\theta,u\) exist, and write \(\theta_0=\pr(\theta)\),
\(u_0=\pr(u)\). Reduction gives \(u_0^Tu_0=1\) and
\(I=\exp(\theta_0\wide{u_0})\), hence \(\theta_0=2k\pi\) for some
\(k\in\mathbb Z\).

For \(i=1,2\), denote by \(\theta_i\) and \(v_i\) the coefficients of \(\varepsilon_i\) in
\(\theta\) and \(u\), respectively. The unit condition gives \(u_0^Tv_i=0\), while
\[
 [\theta u]_{\varepsilon_i}=\theta_i u_0+\theta_0v_i.
\]
The differential of the exponential gives the corresponding first-order coefficient in
right-trivialised axial form:
\begin{equation}\label{eq:tangentmechanism}
 [Q]_{\varepsilon_i}^{\mathrm{ax}}
 =J(\theta_0u_0)(\theta_i u_0+\theta_0v_i).
\end{equation}
Here \([Q]_{\varepsilon_i}^{\mathrm{ax}}\) denotes the axial vector of the
\(\varepsilon_i\)-coefficient of \(Q\) after right trivialisation by \(Q_0^{-1}\). Since
\(Q_0=I\), this trivialisation is the identity. For \(k=0\), \(J(0)=I\) and
\eqref{eq:tangentmechanism} reduces to \(\theta_i u_0\). Thus the two first-order axial
coefficients \(p\) and \(q\) would both be parallel to \(u_0\), contrary to their independence.

Let now \(k\ne0\). Equation \eqref{eq:JacobianConsequences} gives
\[
 J(2k\pi u_0)=u_0u_0^T,
\]
the rank-one projection onto \(\R u_0\). Since \(u_0^Tv_i=0\), formula
\eqref{eq:tangentmechanism} again reduces to \(\theta_i u_0\). It cannot produce the two
independent tangents \(p\) and \(q\). This contradiction excludes every branch.
\end{proof}

\begin{remark}[The ordinary dual algebra]
The mechanism of \cref{thm:noaxis} uses two independent first-order nilpotent directions. For
\(\A=\R[\varepsilon]/(\varepsilon^2)\), every nonzero infinitesimal axial vector
\(\phi=\varepsilon p\) has the unit-axis factorisation
\(\phi=(\varepsilon\|p\|)(p/\|p\|)\). Thus the particular obstruction used in
\cref{thm:noaxis} is unavailable when \(\dim(\m/\m^2)=1\); the theorem makes no general
claim for all Weil algebras with a higher-dimensional first-order space.
\end{remark}

Thus the correct global statement is:
\begin{quote}
Every \(Q\in\SO(3,\A)\) is exponential and satisfies an unnormalised Rodrigues formula;
a representation by one scalar multidual angle and one unit multidual axis exists on a
regular local branch over each nonidentity real fibre, but need not exist over the identity fibre.
\end{quote}

\section{Kinematic consequences and illustrations}\label{sec:kinematics}

The transfer result has three distinct uses in kinematics. The following construction makes the
first two explicit and identifies the multidual differential transform used in higher-order
kinematics \cite{ConduracheOverview2025} as a particular Weil jet.

Let
\[
 \A_n=\R[\varepsilon]/(\varepsilon^{n+1})
\]
be the chain algebra. For a curve \(f:I\to\R^r\) of class \(C^n\), define its multidual
differential transform at \(t\) by
\begin{equation}\label{eq:differentialTransform}
 \breve f(t):=(J_{\A_n}f)(t+\varepsilon)
 =\sum_{k=0}^{n}\frac{\varepsilon^k}{k!}f^{(k)}(t).
\end{equation}
Thus the differential transform is not defined by an infinite operator series: it is the value of
the finite Weil prolongation at \(t+\varepsilon\). In particular,
\cref{thm:functoriality} gives
\[
 \breve{(fg)}=\breve f\,\breve g,
 \qquad
 \breve{(F\circ f)}=(J_{\A_n}F)(\breve f)
\]
whenever the products and compositions are defined.

For example, if \(R\in C^n(I,\SO(3,\R))\) is a rotation curve, then
\[
 \breve R(t)=\sum_{k=0}^{n}\frac{\varepsilon^k}{k!}R^{(k)}(t)
 \in M_3(\A_n).
\]
Prolongation of the real identities \(R^TR=I\) and \(\det R=1\) yields, exactly,
\begin{equation}\label{eq:transformedOrthogonality}
 \breve R^{T}\breve R=I,
 \qquad
 \det\breve R=1;
 \qquad\text{hence}\qquad
 \breve R\in\SO(3,\A_n).
\end{equation}
If \(R(t)=\exp(\wide{\phi(t)})\) with \(\phi\in C^n(I,\R^3)\), then
\cref{thm:Rodrigues} gives the finite identity
\[
 \breve R
 =I+a_{\A_n}(\breve\phi^{T}\breve\phi)\wide{\breve\phi}
   +b_{\A_n}(\breve\phi^{T}\breve\phi)\wide{\breve\phi}^{2}.
\]
Given the jet of \(\phi\), both the higher derivatives encoded by \(\breve R\) and its orthogonality
constraints are obtained simultaneously, without numerical differentiation and without an
infinite series over \(\A_n\).

The same mechanism covers the hyper-multidual quaternion representation. Let
\(\mathbb D=\R[\varepsilon_0]/(\varepsilon_0^2)\) be the dual algebra and
\(\A_n^{\mathrm{HMD}}=\mathbb D\otimes_{\R}\A_n\) the commutative HMD coefficient
algebra. Let \(\mathbb H\) be the real quaternion algebra and regard a \(C^n\) curve
of dual quaternions \(q:I\to\mathbb H\otimes_{\R}\mathbb D\simeq\mathbb D^4\)
as an \(\R^8\)-valued curve. Quaternion conjugation, denoted by \(^{*}\), fixes the
coefficients in \(\mathbb D\), and after prolongation those in
\(\A_n^{\mathrm{HMD}}\). Put
\[
 \breve q(t)=\sum_{k=0}^{n}\frac{\varepsilon^k}{k!}q^{(k)}(t),
\]
which is a quaternion over \(\A_n^{\mathrm{HMD}}\). Quaternion multiplication and conjugation are
polynomial real maps. Therefore, if \(q(t)q(t)^*=1\), identity transfer gives
\begin{equation}\label{eq:HMDunit}
 \breve q(t)\breve q(t)^*=1.
\end{equation}
For a unit dual quaternion \(q\), the Euler--Rodrigues action
\(\Theta(q)v=qvq^*\) is on pure dual quaternions
\(v\in\operatorname{Im}(\mathbb H)\otimes_{\R}\mathbb D\simeq\mathbb D^3\).
Its coefficients are polynomial real functions of those of \(q\), so its quaternionic
and orthogonal-tensor descriptions commute with finite prolongation. Equations
\eqref{eq:transformedOrthogonality} and \eqref{eq:HMDunit} provide the precise finite-jet basis
for the multidual homogeneous-matrix, HMD orthogonal-tensor, and HMD unit-quaternion
constructions used in higher-order rigid-body and multibody kinematics
\cite{ConduracheOverview2025}.

Finally, composition is preserved under the hypotheses of
\cref{thm:functoriality}: compatible real domains, sufficient finite differentiability, and a
single fixed branch for multivalued functions. These are the operative assumptions when
logarithms, fractional powers, or interpolation charts are used.

For the chain algebra, every individual \(Q\in\SO(3,\A_n)\) is a curve jet.
Indeed, by \cref{thm:surjectivity}, write
\(Q=\operatorname{Exp}_{\A_n}(B)\) with
\(B=\sum_{k=0}^{n}\varepsilon^kB_k\), \(B_k\in\so(3,\R)\). The real curve
\[
 \gamma(s)=\exp\!\left(\sum_{k=0}^{n}s^kB_k\right)
\]
satisfies \((J_{\A_n}\gamma)(\varepsilon)=Q\) by \cref{thm:functoriality}.
For a time-dependent family \(Q(t)\), however, pointwise membership in
\(\SO(3,\A_n)\) does not impose the differential compatibility conditions in
\cref{def:holonomicity}. Those conditions express that the entire family is the jet of
one real motion, namely its scalar-part curve.

\section{Conclusions}

The exponential map \(\so(3,\A)\to\SO(3,\A)\) is surjective for every Weil algebra
\(\A\), and every proper orthogonal \(3\times3\) Weil tensor admits an exact unnormalised Rodrigues
form. This does not imply a normalised axis-angle representation: such a representation is
regular on a chosen local branch over each nonidentity real fibre, including rotations by
\(\pi\), but may fail over
the identity. The counterexample in \cref{thm:noaxis} excludes both the principal logarithm
and every branch with scalar angle \(2k\pi\).

The analytic foundation is finite Weil prolongation. A scalar real trace uniquely determines
a sufficiently regular \(\A\)-holomorphic function, and functoriality transports real
functional, matrix, orthogonality, and group identities exactly. Besides supporting the
geometric results, this places multidual differential transforms and HMD quaternion or matrix
representations within one finite-jet framework, with explicit domain and branch conditions.

\section*{Statements and Declarations}

\noindent\textbf{Funding.} The author received no specific funding for this work.

\noindent\textbf{Competing interests.} The author declares no competing interests.

\noindent\textbf{Data availability.} No datasets were generated or analysed in this study.

\noindent\textbf{Author contributions.} The author is solely responsible for the
conceptualisation, mathematical analysis, proofs, and writing of the manuscript.

\end{document}